\documentclass[reqno]{amsart}
\usepackage{amsfonts}
\usepackage{amsmath}
\usepackage{amssymb}
\usepackage{mathrsfs}
\usepackage[colorlinks]{hyperref}
\usepackage{tikz}
\usetikzlibrary{matrix}
\usepackage{algorithm}
\usepackage{algpseudocode}  
\usepackage{algorithmicx}
\usetikzlibrary{shadows,positioning}
\usepackage{graphicx,lscape}
\usepackage{subcaption}
\graphicspath{}
\DeclareGraphicsExtensions{.pdf,.png,.jpg}

\numberwithin{equation}{section}
\theoremstyle{plain}
\newtheorem{thm}{Theorem}[section]
\newtheorem{lem}[thm]{Lemma}

\theoremstyle{definition}

\theoremstyle{remark}
\newtheorem{rem}{Remark}[section]
\newtheorem{defn}{Definition}

\let\Re\relax
\DeclareMathOperator{\Re}{Re}

\begin{document}
\title[\hfil Heat Equation driven by mixed local-nonlocal operators]
{Heat Equation driven by mixed local-nonlocal operators with non-regular space-dependent coefficients}

\author[Arshyn Altybay]{Arshyn Altybay}
\address{
Arshyn Altybay: 
\endgraf
Institute of Mathematics and Mathematical Modeling,
\endgraf
 125 Pushkin str., 050010 Almaty, Kazakhstan
\endgraf
{\it E-mail address} {\rm arshyn.altybay@gmail.com, arshyn.altybay@math.kz}
}

\author[M. Ruzhansky]{Michael Ruzhansky}
\address{
  Michael Ruzhansky:
  \endgraf
  Department of Mathematics: Analysis, Logic and Discrete Mathematics
  \endgraf
  Ghent University, Krijgslaan 281, Building S8, B 9000 Ghent, Belgium
  \endgraf
  and
  \endgraf
  School of Mathematical Sciences
  \endgraf
  Queen Mary University of London, United Kingdom
  \endgraf
  {\it E-mail address} {\rm michael.ruzhansky@ugent.be}
}

\thanks{}


\begin{abstract}
In this paper, we study the Cauchy problem for a heat equation governed by a mixed local--nonlocal diffusion operator with spatially dependent coefficients. For bounded measurable coefficients satisfying uniform positivity conditions and initial data in $H^1(\mathbb R^d)$, we first establish existence and uniqueness of weak solutions in the natural energy space and derive an a priori energy estimate. We then extend the analysis to strictly positive distributional coefficients and compactly supported distributional initial data by means of a Friedrichs-type regularisation procedure. Under suitable moderateness assumptions on the regularising nets, we establish the existence of very weak solutions and prove uniqueness modulo negligible perturbations of admissible regularisations. Finally, under appropriate convergence assumptions on the regularised coefficients and initial data, together with preservation of uniform positivity, we prove consistency of the very weak solution with the corresponding weak solution.

\end{abstract}

\maketitle
\allowdisplaybreaks
\section{Introduction}\label{0}
Motivated by ecological models in which individuals move by a superposition of
(i) local Brownian diffusion and (ii) long-range L\'evy-type jumps,
Dipierro and Valdinoci \cite{DV21} derived evolution equations whose dispersal mechanism is governed by the mixed operator
\begin{equation}\label{eq:mixed-operator}
\mathcal{L}_0 u \;:=\; -\Delta u + (-\Delta)^s u,\qquad s\in(0,1).
\end{equation}
Here the local part $-\Delta$ describes short-range random motion, whereas the fractional Laplacian $(-\Delta)^s$
encodes long-range interactions and jump processes.
In the whole space $\mathbb{R}^d$, $(-\Delta)^s$ may be defined via the principal value integral
\[
(-\Delta)^s u(x)
= c_{d,s}\,\mathrm{P.V.}\!\int_{\mathbb{R}^d}\frac{u(x)-u(y)}{|x-y|^{d+2s}}\,dy,
\]
so that $\mathcal{L}_0$ corresponds (up to sign conventions) to the infinitesimal generator of a stochastic motion
obtained by superposing a Brownian component with an independent symmetric $\alpha$-stable jump process, $\alpha=2s$
\cite{CKS12,CKSV12}.
In bounded domains, the mixed nature of \eqref{eq:mixed-operator} leads naturally to boundary conditions that combine
a classical Neumann constraint for the local part with a genuinely nonlocal flux condition for the fractional component,
as proposed in \cite{DV21}. A systematic PDE analysis of elliptic problems driven by \eqref{eq:mixed-operator}
was later developed in \cite{BDVVe22}.

The mixed local--nonlocal operator $\mathcal{L}_0$ has attracted considerable attention in recent years and has been investigated from several perspectives. This sustained interest is largely driven by its ability to capture mechanisms in which \emph{local diffusion} coexists with \emph{long-range interactions} and jump-like dispersal. In probability theory, operators of the form $-\Delta+(-\Delta)^s$ arise as infinitesimal generators of stochastic dynamics obtained by superimposing a Brownian motion with an independent L\'evy jump process, thereby modelling trajectories with frequent small displacements interspersed with occasional long relocations. In mathematical biology and ecology, the same superposition naturally appears in models of animal movement and optimal foraging, where long-range relocation (often idealised via L\'evy flights) complements local exploratory behaviour and leads to mixed dispersal laws. Similar hybrid effects are also used to describe \emph{anomalous transport} in complex environments (e.g.\ dispersal in porous or turbulent media), \emph{heterogeneous materials} in which short-range diffusion is coupled with nonlocal transfer, and \emph{diffusion with trapping, absorption, or killing}, where lower-order potential terms represent loss mechanisms or reactions. For further background and additional references, we refer the reader to \cite{CH15, DLV23, DLV24, DLV24b, BDVVe25, KT25, YSVM25a, YSVM25b} and the works cited therein.

In this paper, we extend \eqref{eq:mixed-operator} by allowing variable coefficients and consider the Cauchy problem
\begin{equation}\label{eq:1.1}
\left\{
\begin{array}{l}
u_t(t,x) + \mathcal{L}u(t,x)= 0, \qquad (t,x)\in(0,T)\times \mathbb{R}^{d},\\
u(0,x)=u_{0}(x), \qquad x \in \mathbb{R}^{d},
\end{array}
\right.
\end{equation}
where $\mathcal{L}$ is the mixed local--nonlocal operator
\[
\mathcal{L}u
= -\nabla \cdot\bigl(a(x)\nabla u\bigr)
+ (-\Delta)^{s/2}\Bigl(b(x)\,(-\Delta)^{s/2} u\Bigr)
+ c(x)u,
\qquad 0<s<1.
\]
Here $u$ denotes a scalar state variable (e.g., temperature, concentration, density), $u_0$ is the initial datum. The coefficients $a,b,c$ are assumed real-valued, with
$a(x)\ge a_0>0$, $b(x)\ge b_0>0$, and $c(x)\ge c_0 > 0$ for a.e.\ $x\in\mathbb{R}^d$.
Moreover, $(-\Delta)^{s/2}$ denotes the (Riesz) fractional Laplacian.
Note that when $b\equiv \text{const}$, the fractional term reduces to $b\,(-\Delta)^s u$.

More precisely, the divergence-form term $-\nabla\!\cdot(a(x)\nabla u)$ models heterogeneous classical diffusion,
whereas the fractional component $(-\Delta)^{s/2}\bigl(b(x)(-\Delta)^{s/2}u\bigr)$ incorporates nonlocal effects of order $2s$
modulated by $b(x)$. The lower-order term $c(x)u$ may be interpreted, depending on the application, as reaction, damping,
or absorption.

The main purpose of this work is to extend the analysis to coefficients that may be singular in the spatial variable. In the very weak setting, we allow
\[
a,b,c\in\mathcal D'(\mathbb R^d),
\,
u_0\in\mathcal E'(\mathbb R^d),
\]
where the coefficients are assumed to be strictly positive in the sense of distributions. More precisely, there exist constants $(a_0,b_0,c_0>0)$ such that
\[
\langle a-a_0,\varphi\rangle\ge0,\qquad
\langle b-b_0,\varphi\rangle\ge0,\qquad
\langle c-c_0,\varphi\rangle\ge0
\]
for every non-negative test function
($\varphi\in\mathcal D(\mathbb R^d)$).
The regularising nets are additionally required to satisfy the global moderateness conditions introduced in Section~\ref{3}.

For distributional coefficients, products such as
$a\nabla u,\,
b(-\Delta)^{s/2}u,
\,
cu$
cannot in general be interpreted within the classical theory of distributions. This is a manifestation of the well-known obstruction to multiplication of arbitrary distributions identified by Schwartz \cite{Sch54}. Consequently, the equation cannot in general be treated directly by the standard weak formulation.

To overcome this difficulty, we employ the notion of \emph{very weak solutions}, introduced by Garetto and the second author \cite{GR15}. The basic idea is to regularise the distributional coefficients and initial data by convolution with suitable mollifiers. This generates a net of Cauchy problems with smooth coefficients and data, indexed by the regularisation parameter $\varepsilon\in(0,1]$. For each fixed $\varepsilon$, the corresponding regularised problem can be treated by the classical weak theory, yielding a net of solutions $(u_\varepsilon)_\varepsilon$. This net is called a very weak solution when it satisfies suitable growth estimates with respect to $\varepsilon$, referred to as \emph{moderateness}. The notion of very weak solution is inspired by the theory of algebras of generalised functions \cite{GO14,GO15,GHO09,GKOS01,Obe92}. However, in the present approach we work directly with nets of regularised functions and solutions rather than with equivalence classes in a generalised-function algebra. More generally, this net-based formulation provides flexibility in the choice of the regularisation procedure: different coefficients and data
may be regularised using different mollifiers and regularisation scales, depending on the structure of the problem. In the present work, however, we employ a common Friedrichs mollifier and a common regularisation scale for the coefficients and the initial datum.

This approach has proved effective for second-order hyperbolic and parabolic equations with non-regular coefficients \cite{RT17a,RT17b}, and has subsequently been extended to equations with singular potentials, distributional lower-order terms, and related classes of evolution equations \cite{ART19,ARST21a,ARST21b,ARST21c,ARST25,CRT21,SW22,CRT22a,CRT22b,BLO22,RY22,RST23,CDRT23,CDRT24,RSY24,GS24}.

The main contributions of the present paper are as follows. First, for bounded measurable coefficients satisfying uniform positivity conditions, we establish existence and uniqueness of weak solutions and derive an energy estimate in the natural $H^1(\mathbb R^d)$-based framework. Second, for strictly positive distributional coefficients and distributional initial data, we construct very weak solutions by spatial regularisation and prove moderateness of the corresponding solution nets. Third, we establish uniqueness modulo negligible perturbations of admissible regularisations. Finally, we prove consistency with the regular weak theory under suitable convergence assumptions on the regularised coefficients and initial data. 

Thus, the present work provides a unified weak and very weak treatment of a heat equation driven by a variable-coefficient mixed local--nonlocal diffusion operator. To the best of our knowledge, a very weak solution theory for this variable-coefficient mixed local–nonlocal parabolic problem has not previously been developed.

The remainder of this paper is organised as follows.
Section~\ref{1} contains the necessary preliminaries.
In Section~\ref{2}, we establish well-posedness and energy estimates
for the regular problem.
Section~\ref{3} develops the very weak solution framework and proves
existence, uniqueness, and consistency results for spatially irregular
coefficients.
Section~\ref{4} concludes the paper.

\section{Preliminaries}\label{1}
In this section, we recall several standard definitions and auxiliary results used throughout the paper. The material on the Riesz fractional Laplacian, fractional Sobolev spaces, the Gagliardo seminorm, and the equivalence between
Fourier and integral characterisations of fractional norms can be found, for example, in \cite{DNPV12,Kwa17,AF03}. The Fourier transform convention and Plancherel identity are classical; see \cite{SW71,LL01}. The functional-analytic
framework based on Gelfand triples and the Lions--Magenes lemma is standard in the variational theory of parabolic equations; see \cite{LM72,Bre11,Eva10}.

\begin{defn}[\textbf{Riesz fractional Laplacian in $\mathbb R^d$}]\label{def:frac-lap-Rd}
Let $0<s<1$ and $u\in \mathcal S(\mathbb R^d)$. The (Riesz) fractional Laplacian is defined by
\[
(-\Delta)^s u(x)= C_{d,s}\,\mathrm{p.v.}\int_{\mathbb R^d}\frac{u(x)-u(y)}{|x-y|^{d+2s}}\,dy,
\]
where $C_{d,s}>0$ is a normalization constant. Equivalently, in Fourier variables,
\[
\widehat{(-\Delta)^s u}(\xi)=|\xi|^{2s}\,\hat u(\xi).
\]
\end{defn}

\begin{defn}[\textbf{Fourier transform}]\label{def:fourier}
For $f\in L^1(\mathbb R^d)$ we define
\[
\hat f(\xi)=\mathcal F(f)(\xi):=\int_{\mathbb R^d} e^{-ix\cdot \xi} f(x)\,dx,
\qquad
f(x)=\mathcal F^{-1}(\hat f)(x)=\frac{1}{(2\pi)^d}\int_{\mathbb R^d} e^{ix\cdot \xi}\hat f(\xi)\,d\xi.
\]
\end{defn}

\begin{lem}[\textbf{Plancherel identity}]\label{lem:plancherel}
For $f\in L^2(\mathbb R^d)$ we have $\|f\|_{L^2}^2=\frac{1}{(2\pi)^d}\|\hat f\|_{L^2}^2$.
\end{lem}

\begin{defn}[\textbf{Gagliardo seminorm}]\label{def:gagliardo}
Let $0<s<1$. For $u:\mathbb R^d\to\mathbb R$ define
\[
[u]_{H^s(\mathbb R^d)}^2:=\iint_{\mathbb R^d\times \mathbb R^d}\frac{|u(x)-u(y)|^2}{|x-y|^{d+2s}}\,dx\,dy.
\]
\end{defn}

\begin{lem}[\textbf{Equivalence of fractional norms}]\label{lem:Hs-equivalence}
Let $0<s<1$. Then there exists $C=C(d,s)>0$ such that for all $u\in H^s(\mathbb R^d)$,
\[
\|(-\Delta)^{s/2}u\|_{L^2}^2 \simeq [u]_{H^s(\mathbb R^d)}^2,
\qquad
\|u\|_{H^s}^2 \simeq \|u\|_{L^2}^2+[u]_{H^s}^2.
\]
\end{lem}

\begin{lem}[\textbf{Fractional integration by parts}]\label{lem:frac-ibp}
Let $0<s<1$ and $u,v\in H^s(\mathbb R^d)$. Then
\[
\bigl((-\Delta)^{s/2}u,(-\Delta)^{s/2}v\bigr)_{L^2}
=
\langle (-\Delta)^s u, v\rangle_{H^{-s},H^s}.
\]
In particular, $(-\Delta)^s$ is self-adjoint and nonnegative.
\end{lem}

\begin{lem}[\textbf{Fourier domination for $0<s<1$}]\label{lem:fourier-domination}
Let $0<s<1$. Then for all $\xi\in\mathbb R^d$,
\[
|\xi|^{2s}\le 1+|\xi|^2.
\]
Consequently, for every $v\in H^1(\mathbb R^d)$,
\[
\|(-\Delta)^{s/2}v\|_{L^2}^2 \le \|v\|_{H^1}^2.
\]
\end{lem}

\begin{defn}[\textbf{Gelfand triple and duality}]\label{def:gelfand}
Let $V$ be a Hilbert space continuously and densely embedded into $H$ (another Hilbert space), written $V\hookrightarrow H$.
We identify $H$ with its dual $H'$ and obtain the Gelfand triple
\[
V \hookrightarrow H \simeq H' \hookrightarrow V'.
\]
The duality pairing between $V'$ and $V$ is denoted by $\langle \cdot,\cdot\rangle$.
\end{defn}

\begin{lem}[\textbf{Lions--Magenes}]\label{lem:lions-magenes}
Let $V\hookrightarrow H$ be a Gelfand triple. If
\[
u\in L^2(0,T;V), \qquad u_t\in L^2(0,T;V'),
\]
then $u$ admits a representative in $C([0,T];H)$ and for a.e.\ $t$,
\[
\frac{d}{dt}\frac12\|u(t)\|_H^2 = \langle u_t(t),u(t)\rangle_{V',V}.
\]
\end{lem}





\section{Well-posedness of the  problem in the regular case}\label{2}
In this section, we analyse the well-posedness of the Cauchy problem \eqref{eq:1.1} under the assumption that the equation coefficients and initial data possess sufficient regularity.

\subsection{Energy estimate in the regular case}\label{subsec:regular}
We work with the standard variational formulation associated with $\mathcal L$.
Define the bilinear form
\begin{align*}
B(u,v)&:=\int_{\mathbb R^d} a(x)\,\nabla u(x)\cdot \overline{\nabla v(x)}\,dx
+\int_{\mathbb R^d} b(x)\,(-\Delta)^{\frac{s}{2}}u(x)\,\overline{(-\Delta)^{\frac{s}{2}}v(x)}\,dx
\\&+\int_{\mathbb R^d} c(x)\,u(x)\overline{v(x)}\,dx,  
\end{align*}

for $u,v\in H^1(\mathbb R^d)$. Since $0<s<1$, one has the continuous embedding
$H^1(\mathbb R^d)\hookrightarrow H^s(\mathbb R^d)$, hence all terms are well-defined.
Moreover, $B$ is bounded and coercive on $H^1(\mathbb R^d)$ in the sense that
\[
|B(u,v)|
\le
\bigl(\|a\|_{L^\infty}+\|b\|_{L^\infty}+\|c\|_{L^\infty}\bigr)\,
\bigl(\|u\|_{H^1}+\|(-\Delta)^{s/2}u\|_{L^2}\bigr)\,
\bigl(\|v\|_{H^1}+\|(-\Delta)^{s/2}v\|_{L^2}\bigr),
\]
and
\[
B(u,u)\ge a_0\|\nabla u\|_{L^2}^2+b_0\|(-\Delta)^{\frac{s}{2}}u\|_{L^2}^2+c_0\|u\|_{L^2}^2
\quad\text{for all }u\in H^1(\mathbb R^d).
\]

\begin{defn}[Weak solution]\label{def:weak solution}
A function $u$ is called a weak solution of \eqref{eq:1.1} if
\[
u\in L^2(0,T;H^1(\mathbb R^d)),\quad u_t\in L^2\bigl(0,T;(H^1(\mathbb R^d))'\bigr),
\]
$u(0)=u_0$ in $L^2$, and for a.e.\ $t\in(0,T)$,
\begin{equation}\label{eq:weakform-corr}
\langle u_t(t),\varphi\rangle + B(u(t),\varphi)=0
\quad\forall\,\varphi\in H^1(\mathbb R^d).
\end{equation}
\end{defn}

\begin{lem}[Existence and uniqueness]\label{lem:exist-uniq}
Let $0<s<1$ and let $a,b,c\in L^\infty(\mathbb R^d)$ be real-valued with
\[
a(x)\ge a_0>0,\quad b(x)\ge b_0>0,\quad c(x)\ge c_0>0
\quad\text{for a.e. }x\in\mathbb R^d.
\]
Let $u_0\in H^1(\mathbb R^d)$ and consider \eqref{eq:1.1}.
Then there exists a unique weak solution $u$ in the sense of Definition \ref{def:weak solution} such that
\[
u\in L^\infty(0,T;H^1(\mathbb R^d))\cap C([0,T];L^2(\mathbb R^d)),
\quad
u_t\in L^2\bigl(0,T;(H^1(\mathbb R^d))' \bigr).
\]
Moreover, for a.e.\ $t\in[0,T]$,
\begin{equation}\label{eq:apriori-final-Linfty-corr}
\|u(t)\|_{L^2(\mathbb R^d)}^2
+\|\nabla u(t)\|_{L^2(\mathbb R^d)}^2
+\|(-\Delta)^{\frac{s}{2}}u(t)\|_{L^2(\mathbb R^d)}^2
\le
C\,\|u_0\|_{H^1(\mathbb R^d)}^2,
\end{equation}
with
\[
C:=2\Bigl(1+\frac{1}{a_0}+\frac{1}{b_0}\Bigr)
\Bigl(1+\|a\|_{L^\infty(\mathbb R^d)}+\|b\|_{L^\infty(\mathbb R^d)}+\|c\|_{L^\infty(\mathbb R^d)}\Bigr).
\]
\end{lem}

\begin{proof}

\medskip
\noindent\textbf{Galerkin approximation.}
Let $\{w_k\}_{k\ge1}\subset H^1(\mathbb R^d)$ be a complete orthonormal
system in $L^2(\mathbb R^d)$ such that, with
$V_N:=\mathrm{span}\{w_1,\dots,w_N\},$
we have
\[
\overline{\bigcup_{N=1}^{\infty}V_N}^{\,H^1(\mathbb R^d)}
=
H^1(\mathbb R^d).
\]
Choose $u_0^N\in V_N$ such that
$u_0^N\to u_0
\quad\text{in }H^1(\mathbb R^d).$

Seek $u^N:[0,T]\to V_N$ solving
\begin{equation}\label{eq:Galerkin-corr}
(u_t^N(t),\varphi)_{L^2} + B(u^N(t),\varphi)=0,
\qquad \forall \varphi\in V_N,\ \text{a.e. }t\in(0,T),
\qquad u^N(0)=u_0^N.
\end{equation}
Writing $u^N(t)=\sum_{k=1}^N g_k(t)w_k$, \eqref{eq:Galerkin-corr} is an ODE system for $g_k$ with a (symmetric) positive definite matrix.
Hence there exists a unique $u^N\in C^1([0,T];V_N)$.

\medskip
\noindent\textit{Step 1: Energy identities (performed for $u^N$).}

\smallskip
\noindent\emph{(i) Testing with $\varphi=u^N(t)$.}
Taking $\varphi=u^N(t)$ in \eqref{eq:Galerkin-corr} gives
\[
(u_t^N(t),u^N(t))_{L^2}+B(u^N(t),u^N(t))=0.
\]
Taking real parts, we obtain
\[
\Re (u_t^N(t),u^N(t))_{L^2}+B(u^N(t),u^N(t))=0.
\]
First,
\begin{align*}
\Re (u_t^N(t),u^N(t))_{L^2}
&=\Re\int_{\mathbb R^d} u_t^N(t,x)\,\overline{u^N(t,x)}\,dx
=\frac12\frac{d}{dt}\int_{\mathbb R^d} |u^N(t,x)|^2\,dx
\\&=\frac12\frac{d}{dt}\|u^N(t)\|_{L^2}^2. 
\end{align*}
Next, by definition of $B$,
\begin{align*}
B(u^N(t),u^N(t))
&=\int_{\mathbb R^d} a(x)\,|\nabla u^N(t,x)|^2\,dx
+\int_{\mathbb R^d} b(x)\,\bigl|(-\Delta)^{\frac{s}{2}}u^N(t,x)\bigr|^2\,dx
\\&+\int_{\mathbb R^d} c(x)\,|u^N(t,x)|^2\,dx,
\end{align*}
hence
\[
B(u^N(t),u^N(t))
=\|a^{1/2}\nabla u^N(t)\|_{L^2}^2
+\|b^{1/2}(-\Delta)^{\frac{s}{2}}u^N(t)\|_{L^2}^2
+\|c^{1/2}u^N(t)\|_{L^2}^2.
\]
Combining these identities yields
\begin{equation}\label{eq:L2-energyN-corr}
\frac12\frac{d}{dt}\|u^N(t)\|_{L^2}^2
+\|a^{1/2}\nabla u^N(t)\|_{L^2}^2
+\|b^{1/2}(-\Delta)^{\frac{s}{2}}u^N(t)\|_{L^2}^2
+\|c^{1/2}u^N(t)\|_{L^2}^2
=0.
\end{equation}

\smallskip
\noindent\emph{(ii) Testing with $\varphi=u_t^N(t)$.}
Taking $\varphi=u_t^N(t)$ in \eqref{eq:Galerkin-corr} gives
\[
(u_t^N(t),u_t^N(t))_{L^2}+B(u^N(t),u_t^N(t))=0,
\]
that is,
\[
\|u_t^N(t)\|_{L^2}^2 + B(u^N(t),u_t^N(t))=0.
\]
Taking real parts, we obtain
\[
\|u_t^N(t)\|_{L^2}^2 + \Re B(u^N(t),u_t^N(t))=0.
\]
We again expand $B(u^N,u_t^N)$ term-by-term:
\begin{align*}
&B(u^N(t),u_t^N(t))
=\int_{\mathbb R^d} a(x)\,\nabla u^N(t,x)\cdot \overline{\nabla u_t^N(t,x)}\,dx
\\
&\quad+\int_{\mathbb R^d} b(x)\,(-\Delta)^{\frac{s}{2}}u^N(t,x)\,
\overline{(-\Delta)^{\frac{s}{2}}u_t^N(t,x)}\,dx
+\int_{\mathbb R^d} c(x)\,u^N(t,x)\,\overline{u_t^N(t,x)}\,dx.
\end{align*}
Since $a,b,c$ are independent of $t$, each term is an exact time derivative (in the real part):
\[
\Re\int_{\mathbb R^d} a\,\nabla u^N \cdot \overline{\nabla u_t^N}\,dx
=\frac12\frac{d}{dt}\int_{\mathbb R^d} a\,|\nabla u^N|^2\,dx
=\frac12\frac{d}{dt}\|a^{1/2}\nabla u^N\|_{L^2}^2,
\]
\[
\Re\int_{\mathbb R^d} b\,(-\Delta)^{\frac{s}{2}}u^N \, \overline{(-\Delta)^{\frac{s}{2}}u_t^N}\,dx
=\frac12\frac{d}{dt}\int_{\mathbb R^d} b\,\bigl|(-\Delta)^{\frac{s}{2}}u^N\bigr|^2\,dx
=\frac12\frac{d}{dt}\|b^{1/2}(-\Delta)^{\frac{s}{2}}u^N\|_{L^2}^2,
\]
\[
\Re\int_{\mathbb R^d} c\,u^N \overline{u_t^N}\,dx
=\frac12\frac{d}{dt}\int_{\mathbb R^d} c\,|u^N|^2\,dx
=\frac12\frac{d}{dt}\|c^{1/2}u^N\|_{L^2}^2.
\]
Therefore,
\[
\Re B(u^N(t),u_t^N(t))
=\frac12\frac{d}{dt}\Bigl[
\|a^{1/2}\nabla u^N(t)\|_{L^2}^2
+\|b^{1/2}(-\Delta)^{\frac{s}{2}}u^N(t)\|_{L^2}^2
+\|c^{1/2}u^N(t)\|_{L^2}^2
\Bigr],
\]
and we arrive at
\begin{equation}\label{eq:L2-energy2N-corr}
\|u_t^N(t)\|_{L^2}^2
+\frac12\frac{d}{dt}\Bigl[
\|a^{1/2}\nabla u^N(t)\|_{L^2}^2
+\|b^{1/2}(-\Delta)^{\frac{s}{2}}u^N(t)\|_{L^2}^2
+\|c^{1/2}u^N(t)\|_{L^2}^2
\Bigr]
=0.
\end{equation}

Summing \eqref{eq:L2-energyN-corr} and \eqref{eq:L2-energy2N-corr} yields
\[
\frac12\frac{d}{dt}E_N(t)
=-
\Bigl[\|u_t^N(t)\|_{L^2}^2
+\|a^{1/2}\nabla u^N(t)\|_{L^2}^2
+\|b^{1/2}(-\Delta)^{\frac{s}{2}}u^N(t)\|_{L^2}^2
+\|c^{1/2}u^N(t)\|_{L^2}^2\Bigr]
\le 0,
\]
where
\[
E_N(t):=\|u^N(t)\|_{L^2}^2
+\|a^{1/2}\nabla u^N(t)\|_{L^2}^2
+\|b^{1/2}(-\Delta)^{\frac{s}{2}}u^N(t)\|_{L^2}^2
+\|c^{1/2}u^N(t)\|_{L^2}^2.
\]
Hence $E_N(t)\le E_N(0)$ for all $t\in[0,T]$.

\medskip
\noindent\textit{Step 2: From weighted to standard norms.}
Using $a\ge a_0$ and $b\ge b_0$, we obtain
\[
\|u^N(t)\|_{L^2}^2+\|\nabla u^N(t)\|_{L^2}^2+\|(-\Delta)^{\frac{s}{2}}u^N(t)\|_{L^2}^2
\le \Bigl(1+\frac{1}{a_0}+\frac{1}{b_0}\Bigr)E_N(t)
\le \Bigl(1+\frac{1}{a_0}+\frac{1}{b_0}\Bigr)E_N(0).
\]

\medskip
\noindent\textit{Step 3: Estimate $E_N(0)$ by $\|u_0^N\|_{H^1}^2$.}
By $a,b,c\in L^\infty(\mathbb R^d)$,
\[
E_N(0)
\le \bigl(1+\|a\|_{L^\infty}+\|b\|_{L^\infty}+\|c\|_{L^\infty}\bigr)
\Bigl(\|u_0^N\|_{L^2}^2+\|\nabla u_0^N\|_{L^2}^2+\|(-\Delta)^{\frac{s}{2}}u_0^N\|_{L^2}^2\Bigr).
\]
Moreover, for $0<s<1$ one has the pointwise Fourier inequality $|\xi|^{2s}\le 1+|\xi|^2$, hence
\[
\|(-\Delta)^{\frac{s}{2}}v\|_{L^2}^2 \le \|v\|_{H^1}^2
\qquad\text{for all }v\in H^1(\mathbb R^d),
\]
and therefore
\[
\|v\|_{L^2}^2+\|\nabla v\|_{L^2}^2+\|(-\Delta)^{\frac{s}{2}}v\|_{L^2}^2
\le 2\,\|v\|_{H^1}^2
\qquad\text{for all }v\in H^1(\mathbb R^d).
\]
Consequently,
\[
E_N(0)\le 2\bigl(1+\|a\|_{L^\infty}+\|b\|_{L^\infty}+\|c\|_{L^\infty}\bigr)\,\|u_0^N\|_{H^1}^2.
\]
Combining with Step~2 yields, for all $t\in[0,T]$,
\[
\|u^N(t)\|_{L^2}^2+\|\nabla u^N(t)\|_{L^2}^2+\|(-\Delta)^{\frac{s}{2}}u^N(t)\|_{L^2}^2
\le C\,\|u_0^N\|_{H^1}^2,
\]
with $C$ exactly as in the statement. In particular, $\{u^N\}_N$ is bounded in
$L^\infty(0,T;H^1(\mathbb R^d))$.

\medskip
\noindent\textit{Step 4: Passage $N\to\infty$ (weak compactness and identification of the limit).}
From \eqref{eq:L2-energy2N-corr} and $E_N(t)\le E_N(0)$ we also have
\[
\int_0^T \|u_t^N(t)\|_{L^2}^2\,dt
\le \frac12\Bigl(
\|a^{1/2}\nabla u_0^N\|_{L^2}^2
+\|b^{1/2}(-\Delta)^{\frac{s}{2}}u_0^N\|_{L^2}^2
+\|c^{1/2}u_0^N\|_{L^2}^2
\Bigr)
\le C_1\|u_0\|_{H^1}^2
\]
for some constant $C_1$ independent of $N$ (by the bounds in Step~3 and $u_0^N\to u_0$ in $H^1$).
Hence $\{u_t^N\}_N$ is bounded in $L^2(0,T;L^2(\mathbb R^d))$, and therefore also in
$L^2\bigl(0,T;(H^1(\mathbb R^d))'\bigr)$.

By Banach--Alaoglu, there exists a subsequence (not relabeled) and a function $u$ such that
\[
u^N \rightharpoonup^\ast u \ \text{ in }L^\infty(0,T;H^1(\mathbb R^d)),
\quad
u_t^N \rightharpoonup u_t \ \text{ in }L^2\bigl(0,T;(H^1(\mathbb R^d))'\bigr).
\]
Passing to the limit in \eqref{eq:Galerkin-corr} (written in integrated form against time-dependent test functions in $L^2(0,T;H^1)$)
yields that $u$ satisfies the weak formulation \eqref{eq:weakform-corr}.

Moreover, since
$u^N\rightharpoonup u
\quad\text{in }L^2(0,T;H^1(\mathbb R^d)),$
and
$u_t^N\rightharpoonup u_t
\quad\text{in }L^2(0,T;(H^1(\mathbb R^d))'),$
the Lions--Magenes lemma yields
$u\in C([0,T];L^2(\mathbb R^d)).$
Furthermore, the trace map
$v\mapsto v(0)$
is continuous from
\[
\{v\in L^2(0,T;H^1):v_t\in L^2(0,T;(H^1)')\}
\]
into $L^2(\mathbb R^d)$. Hence
$u^N(0)\rightharpoonup u(0)
\quad\text{in }L^2(\mathbb R^d).$
Since
$u^N(0)=u_0^N
\quad\text{and}\quad
u_0^N\to u_0
\quad\text{strongly in }L^2(\mathbb R^d),$
we conclude that
$u(0)=u_0.$

Finally, define the equivalent Hilbert norm on $H^1(\mathbb R^d)$ by
\[
\|v\|_X^2
:=
\|v\|_{L^2(\mathbb R^d)}^2
+
\|\nabla v\|_{L^2(\mathbb R^d)}^2
+
\|(-\Delta)^{s/2}v\|_{L^2(\mathbb R^d)}^2.
\]
Since $0<s<1$, the norm $\|\cdot\|_X$ is equivalent to the standard
$H^1(\mathbb R^d)$ norm. Hence the estimate obtained in Step~3 shows that
$(u^N)_N$ is bounded in $L^\infty(0,T;X)$. By weak-* lower
semicontinuity,
\[
\|u\|_{L^\infty(0,T;X)}
\le
\liminf_{N\to\infty}
\|u^N\|_{L^\infty(0,T;X)}.
\]
Since $u_0^N\to u_0$ in $H^1(\mathbb R^d)$, passing to the limit in the
estimate from Step~3 yields
\[
\operatorname*{ess\,sup}_{t\in[0,T]}
\left(
\|u(t)\|_{L^2}^2
+
\|\nabla u(t)\|_{L^2}^2
+
\|(-\Delta)^{s/2}u(t)\|_{L^2}^2
\right)
\le
C\|u_0\|_{H^1}^2.
\]
Thus \eqref{eq:apriori-final-Linfty-corr} follows.

\medskip
\noindent\textbf{Uniqueness.}
Let $u_1,u_2$ be two weak solutions and set $w:=u_1-u_2$.
Then $w(0)=0$ and $\langle w_t,\varphi\rangle + B(w,\varphi)=0$ for all $\varphi\in H^1$.
Taking $\varphi=w(t)$ gives
\[
\frac12\frac{d}{dt}\|w(t)\|_{L^2}^2 + B(w(t),w(t))=0.
\]
Since $B(w,w)\ge 0$, we have $\frac{d}{dt}\|w(t)\|_{L^2}^2\le 0$ and $\|w(0)\|_{L^2}=0$,
hence $\|w(t)\|_{L^2}=0$ for all $t\in[0,T]$, i.e.\ $u_1\equiv u_2$.
\end{proof}

\section{Very weak well-posedness}\label{3}

In this section, we extend the notion of solution to the case where the
spatial coefficients and the initial datum are irregular (distributional)
on $\mathbb R^d$.

More precisely, we assume that
$a,b,c\in\mathcal D'(\mathbb R^d),
\quad
u_0\in\mathcal E'(\mathbb R^d),$
and that the coefficients are strictly positive in the sense of
distributions: there exist constants $a_0,b_0,c_0>0$ such that
$\langle a-a_0,\varphi\rangle\ge0,
\quad
\langle b-b_0,\varphi\rangle\ge0,
\quad
\langle c-c_0,\varphi\rangle\ge0$
for every $\varphi\in\mathcal D(\mathbb R^d)$ with $\varphi\ge0$.

Equivalently,
$a=a_0+\mu_a,
\quad
b=b_0+\mu_b,
\quad
c=c_0+\mu_c,$
where $\mu_a,\mu_b,\mu_c$ are non-negative distributions and hence
positive Radon measures.

We interpret \eqref{eq:1.1} in a very weak sense by means of a
Friedrichs-type regularisation procedure. The guiding idea is to replace
$a,b,c,u_0$ by suitable regularising nets and solve a family of regular
problems, for which the regular theory and the a priori estimates from
Section~\ref{2} apply.

\subsection{Spatial regularisation and moderate nets}
\label{subsec:spatial-reg}

Let $\psi\in C_c^\infty(\mathbb R^d)$ be a standard Friedrichs mollifier,
i.e.
$\psi\ge0,
\quad
\int_{\mathbb R^d}\psi(x)\,dx=1.$

Let $\omega:(0,1]\to(0,1]$
be a fixed regularisation scale such that
$\omega(\varepsilon)\to0
\quad
\text{as }\varepsilon\to0^+.$

For $\varepsilon\in(0,1]$, we set
$\psi_{\omega(\varepsilon)}(x)
:=
\omega(\varepsilon)^{-d}
\psi\left(
\frac{x}{\omega(\varepsilon)}
\right),
\quad
x\in\mathbb R^d.$

We regularise in the spatial variable by convolution:
\begin{equation}\label{eq:spatial-reg-defs}
a_\varepsilon
:=
a*\psi_{\omega(\varepsilon)},
\quad
b_\varepsilon
:=
b*\psi_{\omega(\varepsilon)},
\quad
c_\varepsilon
:=
c*\psi_{\omega(\varepsilon)},
\quad
u_{0,\varepsilon}
:=
u_0*\psi_{\omega(\varepsilon)}.
\end{equation}
Then
$a_\varepsilon,
b_\varepsilon,
c_\varepsilon,
u_{0,\varepsilon}
\in C^\infty(\mathbb R^d)$
for every fixed $\varepsilon\in(0,1]$.

Moreover, since $a-a_0$, $b-b_0$, and $c-c_0$ are non-negative
distributions and $\psi_{\omega(\varepsilon)}\ge0$, we have
\begin{equation}\label{eq:uniform-positivity-eps}
a_\varepsilon(x)\ge a_0,
\quad
b_\varepsilon(x)\ge b_0,
\quad
c_\varepsilon(x)\ge c_0,
\quad
x\in\mathbb R^d.
\end{equation}

For each $\varepsilon\in(0,1]$, we consider the regularised Cauchy problem
\begin{equation}\label{eq:reg-problem-space}
\left\{
\begin{aligned}
\partial_t u_\varepsilon(t,x)
+
\mathcal L_\varepsilon u_\varepsilon(t,x)
&=0,
&&
(t,x)\in(0,T)\times\mathbb R^d,
\\
u_\varepsilon(0,x)
&=
u_{0,\varepsilon}(x),
&&
x\in\mathbb R^d,
\end{aligned}
\right.
\end{equation}
where
\begin{equation}\label{eq:L-eps-space}
\mathcal L_\varepsilon u
:=
-\nabla\cdot
\bigl(
a_\varepsilon(x)\nabla u
\bigr)
+
(-\Delta)^{s/2}
\left(
b_\varepsilon(x)(-\Delta)^{s/2}u
\right)
+
c_\varepsilon(x)u.
\end{equation}

To control the behaviour as $\varepsilon\to0$, we use the notions of
moderateness and negligibility.

\begin{defn}[Moderate nets]\label{def:moderate-space}
Let
$\omega:(0,1]\to(0,1]$
be a fixed regularisation scale such that
$\omega(\varepsilon)\to0$ as $\varepsilon\to0^+$.

\begin{enumerate}

\item[(i)]
A net $(g_\varepsilon)_\varepsilon$ is said to be
$L^\infty(\mathbb R^d)$-moderate if there exist constants
$C>0$ and $N\in\mathbb N_0$ such that
\[
\|g_\varepsilon\|_{L^\infty(\mathbb R^d)}
\le
C\,\omega(\varepsilon)^{-N},
\quad
\varepsilon\in(0,1].
\]

\item[(ii)]
A net $(r_\varepsilon)_\varepsilon$ is said to be
$H^1(\mathbb R^d)$-moderate if there exist constants
$C>0$ and $N\in\mathbb N_0$ such that
\[
\|r_\varepsilon\|_{H^1(\mathbb R^d)}
\le
C\,\omega(\varepsilon)^{-N},
\quad
\varepsilon\in(0,1].
\]

\item[(iii)]
A net $(u_\varepsilon)_\varepsilon$ is said to be
$L^\infty(0,T;H^1(\mathbb R^d))$-moderate if there exist constants
$C>0$ and $N\in\mathbb N_0$ such that
\[
\|u_\varepsilon\|_{L^\infty(0,T;H^1(\mathbb R^d))}
\le
C\,\omega(\varepsilon)^{-N},
\quad
\varepsilon\in(0,1].
\]

\end{enumerate}
\end{defn}

\begin{rem}
Polynomial moderateness is natural for compactly supported distributions.

Indeed, let
$T\in\mathcal E'(\mathbb R^d)$
be a distribution of finite order $m$. Then there exists $C>0$ such that
for every test function $\varphi$,
\[
|\langle T,\varphi\rangle|
\le
C
\max_{|\alpha|\le m}
\|\partial^\alpha\varphi\|_{L^\infty}.
\]
Consequently, for every multi-index $\beta$,
\[
\partial^\beta
\bigl(
T*\psi_{\omega(\varepsilon)}
\bigr)(x)
=
\left\langle
T,
\partial^\beta
\psi_{\omega(\varepsilon)}(x-\cdot)
\right\rangle.
\]
Since
\[
\partial^\gamma
\psi_{\omega(\varepsilon)}(x)
=
\omega(\varepsilon)^{-d-|\gamma|}
(\partial^\gamma\psi)
\left(
\frac{x}{\omega(\varepsilon)}
\right),
\]
we obtain
\[
\left|
\partial^\beta
\bigl(
T*\psi_{\omega(\varepsilon)}
\bigr)(x)
\right|
\le
C_\beta
\omega(\varepsilon)^{-d-m-|\beta|}.
\]

Moreover, since both $T$ and $\psi$ are compactly supported, the supports
of $T*\psi_{\omega(\varepsilon)}$ remain contained in a fixed compact set
for $\varepsilon\in(0,1]$. Hence, for suitable constants
$C>0$ and $N\in\mathbb N_0$,
\[
\|T*\psi_{\omega(\varepsilon)}\|_{L^\infty(\mathbb R^d)}
+
\|T*\psi_{\omega(\varepsilon)}\|_{H^1(\mathbb R^d)}
\le
C\,\omega(\varepsilon)^{-N}.
\]

In particular, if
$a=a_0+\mu_a,
\quad
\mu_a\in\mathcal E'(\mathbb R^d),
\quad
\mu_a\ge0,$
then
$a_\varepsilon
=
a_0+\mu_a*\psi_{\omega(\varepsilon)}
\ge
a_0,$
and $(a_\varepsilon)_\varepsilon$ is
$L^\infty(\mathbb R^d)$-moderate.
The same observation applies to $b$ and $c$.
\end{rem}

\subsection{Existence of very weak solutions}
\label{subsec:existence-VWS-space}

In this subsection, we establish existence of very weak solutions for
\eqref{eq:1.1} with spatially distributional coefficients and initial data.

\begin{defn}[Very weak solution]\label{def:VWS-space}
Let
$u_0\in\mathcal E'(\mathbb R^d),$
and let $a,b,c\in\mathcal D'(\mathbb R^d)$
be strictly positive distributions in the sense described above.

A net $(u_\varepsilon)_\varepsilon$ is called a very weak solution of
\eqref{eq:1.1} if there exist regularising nets
$(a_\varepsilon)_\varepsilon,
\,
(b_\varepsilon)_\varepsilon,
\,
(c_\varepsilon)_\varepsilon,
\,
(u_{0,\varepsilon})_\varepsilon$
such that:

\begin{enumerate}

\item[(i)]
$(a_\varepsilon)_\varepsilon$,
$(b_\varepsilon)_\varepsilon$, and
$(c_\varepsilon)_\varepsilon$ are
$L^\infty(\mathbb R^d)$-moderate regularisations of $a,b,c$, respectively,
and $(u_{0,\varepsilon})_\varepsilon$ is an
$H^1(\mathbb R^d)$-moderate regularisation of $u_0$;

\item[(ii)]
there exist constants $a_0,b_0,c_0>0$, independent of $\varepsilon$, such
that
$a_\varepsilon(x)\ge a_0,
\,
b_\varepsilon(x)\ge b_0,
\,
c_\varepsilon(x)\ge c_0$
for a.e. $x\in\mathbb R^d$ and all $\varepsilon\in(0,1]$;

\item[(iii)]
for every $\varepsilon\in(0,1]$, the function $u_\varepsilon$ is the
unique weak solution of the regularised problem
\eqref{eq:reg-problem-space}--\eqref{eq:L-eps-space}, with
$u_\varepsilon(0)=u_{0,\varepsilon};$

\item[(iv)]
the solution net $(u_\varepsilon)_\varepsilon$ is
$L^\infty(0,T;H^1(\mathbb R^d))$-moderate, i.e. there exist constants
$C>0$ and $N\in\mathbb N_0$ such that
\[
\|u_\varepsilon\|_{L^\infty(0,T;H^1(\mathbb R^d))}
\le
C\,\omega(\varepsilon)^{-N},
\quad
\varepsilon\in(0,1].
\]

\end{enumerate}
\end{defn}

\begin{thm}[Existence of very weak solutions]
\label{thm:existence-VWS-space}

Let  $u_0\in\mathcal E'(\mathbb R^d),$
and let
$a,b,c\in\mathcal D'(\mathbb R^d)$
be strictly positive in the sense of distributions, i.e. there exist
constants $a_0,b_0,c_0>0$ such that
$\langle a-a_0,\varphi\rangle\ge0,
\quad
\langle b-b_0,\varphi\rangle\ge0,
\quad
\langle c-c_0,\varphi\rangle\ge0$
for every
$\varphi\in\mathcal D(\mathbb R^d),
\quad
\varphi\ge0.$

Let 
$(a_\varepsilon,b_\varepsilon,c_\varepsilon,
u_{0,\varepsilon})_\varepsilon$
be the spatial regularisations defined by
\eqref{eq:spatial-reg-defs}.

Assume that
$(a_\varepsilon)_\varepsilon$,
$(b_\varepsilon)_\varepsilon$, and
$(c_\varepsilon)_\varepsilon$
are $L^\infty(\mathbb R^d)$-moderate and that
$(u_{0,\varepsilon})_\varepsilon$ is
$H^1(\mathbb R^d)$-moderate.

Then the Cauchy problem \eqref{eq:1.1} admits a very weak solution
$(u_\varepsilon)_\varepsilon$ in the sense of
Definition~\ref{def:VWS-space}.
\end{thm}

\begin{proof}
Fix $\varepsilon\in(0,1]$.

The regularisations in \eqref{eq:spatial-reg-defs} are smooth. By the
$L^\infty(\mathbb R^d)$-moderateness assumption,
$a_\varepsilon,
b_\varepsilon,
c_\varepsilon
\in
L^\infty(\mathbb R^d),$
and by the $H^1(\mathbb R^d)$-moderateness assumption,
$u_{0,\varepsilon}\in H^1(\mathbb R^d).$

Moreover, since the mollifier is non-negative and
$a-a_0$, $b-b_0$, and $c-c_0$ are non-negative distributions,
$a_\varepsilon(x)-a_0
=
\left\langle
a-a_0,
\psi_{\omega(\varepsilon)}(x-\cdot)
\right\rangle
\ge0.$

Similarly,
$b_\varepsilon(x)\ge b_0,
\,
c_\varepsilon(x)\ge c_0.$
Hence,
$a_\varepsilon\ge a_0,
\,
b_\varepsilon\ge b_0,
\,
c_\varepsilon\ge c_0.$

By the regular well-posedness result of
Lemma~\ref{lem:exist-uniq}, applied with coefficients
$a_\varepsilon,b_\varepsilon,c_\varepsilon$, there exists a unique weak
solution $u_\varepsilon$ of
\eqref{eq:reg-problem-space}--\eqref{eq:L-eps-space}.
In particular,
\[
u_\varepsilon
\in
L^\infty(0,T;H^1(\mathbb R^d))
\cap
C([0,T];L^2(\mathbb R^d)).
\]

We now prove the moderateness of $(u_\varepsilon)_\varepsilon$.

Applying the a priori estimate from Lemma~\ref{lem:exist-uniq}, we obtain
for every $t\in[0,T]$,
\begin{align}
&
\|u_\varepsilon(t)\|_{L^2(\mathbb R^d)}^2
+
\|\nabla u_\varepsilon(t)\|_{L^2(\mathbb R^d)}^2
+
\|(-\Delta)^{s/2}u_\varepsilon(t)\|_{L^2(\mathbb R^d)}^2
\nonumber\\
&\qquad\le
C_\varepsilon
\|u_{0,\varepsilon}\|_{H^1(\mathbb R^d)}^2,
\label{eq:apriori-eps-corr}
\end{align}
where one may take
\begin{equation}\label{eq:Ce-corr}
C_\varepsilon
=
2
\left(
1+\frac{1}{a_0}+\frac{1}{b_0}
\right)
\left(
1
+
\|a_\varepsilon\|_{L^\infty}
+
\|b_\varepsilon\|_{L^\infty}
+
\|c_\varepsilon\|_{L^\infty}
\right).
\end{equation}

By the $L^\infty(\mathbb R^d)$-moderateness of the coefficient nets,
there exist constants $C>0$ and $N\in\mathbb N_0$ such that
\[
\|a_\varepsilon\|_{L^\infty}
+
\|b_\varepsilon\|_{L^\infty}
+
\|c_\varepsilon\|_{L^\infty}
\le
C\,\omega(\varepsilon)^{-N}.
\]
Hence, after changing $C$ and $N$ if necessary,
\[
C_\varepsilon
\le
C\,\omega(\varepsilon)^{-N}.
\]

Moreover, since $(u_{0,\varepsilon})_\varepsilon$ is
$H^1(\mathbb R^d)$-moderate, there exist constants
$C'>0$ and $N'\in\mathbb N_0$ such that
\[
\|u_{0,\varepsilon}\|_{H^1(\mathbb R^d)}^2
\le
C'\,\omega(\varepsilon)^{-N'}.
\]

Combining the above estimates yields
\[
\sup_{t\in[0,T]}
\|u_\varepsilon(t)\|_{H^1(\mathbb R^d)}^2
\le
\widetilde C\,
\omega(\varepsilon)^{-\widetilde N}
\]
for suitable constants
$\widetilde C>0$ and
$\widetilde N\in\mathbb N_0$.

Therefore,
\[
\|u_\varepsilon\|_{L^\infty(0,T;H^1(\mathbb R^d))}
\le
\widehat C
\omega(\varepsilon)^{-\widehat N}
\]
for suitable constants
$\widehat C>0$ and
$\widehat N\in\mathbb N_0$.

Thus $(u_\varepsilon)_\varepsilon$ is a very weak solution of
\eqref{eq:1.1}.
\end{proof}

\subsection{Uniqueness}\label{subsec:uniq-negl-space}

Before stating the uniqueness result, we introduce negligible nets.

\begin{defn}[Negligible nets]\label{def:negligible-space}
Let $X$ be a Banach space.
A net $(g_\varepsilon)_\varepsilon\subset X$ is called
$X$-negligible if, for every $q\in\mathbb N$, there exists
$C_q>0$ such that
\[
\|g_\varepsilon\|_X
\le
C_q\,\omega(\varepsilon)^q,
\quad
\varepsilon\in(0,1].
\]
\end{defn}

\begin{rem}\label{rem:ideal-space}
The product of a negligible net and a moderate net remains negligible
whenever the corresponding multiplication is continuous in the spaces
under consideration.

In particular, if $(n_\varepsilon)_\varepsilon$ is negligible in
$L^\infty(\mathbb R^d)$ and $(m_\varepsilon)_\varepsilon$ is moderate in
$L^2(\mathbb R^d)$, then
\[
\|n_\varepsilon m_\varepsilon\|_{L^2}
\le
\|n_\varepsilon\|_{L^\infty}
\|m_\varepsilon\|_{L^2},
\]
and hence
$(n_\varepsilon m_\varepsilon)_\varepsilon$
is negligible in $L^2(\mathbb R^d)$.

The same argument applies when $m_\varepsilon$ is replaced by
$\nabla m_\varepsilon$ or
$(-\Delta)^{s/2}m_\varepsilon$, provided the corresponding nets are
moderate.
\end{rem}

\begin{thm}[Uniqueness of very weak solutions]
\label{thm:uniq-VWS-space}

Let $a,b,c$ and $u_0$ satisfy the assumptions of
Definition~\ref{def:VWS-space}.

Let $(a_\varepsilon,b_\varepsilon,c_\varepsilon)_\varepsilon$
and
$(\widetilde a_\varepsilon,
\widetilde b_\varepsilon,
\widetilde c_\varepsilon)_\varepsilon$
be two $L^\infty(\mathbb R^d)$-moderate regularisations of
$(a,b,c)$ such that
$(a_\varepsilon-\widetilde a_\varepsilon)_\varepsilon,
\quad
(b_\varepsilon-\widetilde b_\varepsilon)_\varepsilon,
\quad
(c_\varepsilon-\widetilde c_\varepsilon)_\varepsilon$
are negligible in $L^\infty(\mathbb R^d)$.

Let
$(u_{0,\varepsilon})_\varepsilon,
\quad
(\widetilde u_{0,\varepsilon})_\varepsilon$
be two $H^1(\mathbb R^d)$-moderate regularisations of $u_0$ such that
$(u_{0,\varepsilon}
-
\widetilde u_{0,\varepsilon})_\varepsilon$
is negligible in $H^1(\mathbb R^d)$.

Assume furthermore that
$a_\varepsilon\ge a_0>0,
\quad
b_\varepsilon\ge b_0>0,
\quad
c_\varepsilon\ge c_0>0,$
and
$\widetilde a_\varepsilon\ge a_0>0,
\quad
\widetilde b_\varepsilon\ge b_0>0,
\quad
\widetilde c_\varepsilon\ge c_0>0$
for every $\varepsilon\in(0,1]$.

Let $u_\varepsilon$ and $\widetilde u_\varepsilon$ be the corresponding
regular weak solutions.

If both solution nets are
$L^\infty(0,T;H^1(\mathbb R^d))$-moderate, then
$(u_\varepsilon-\widetilde u_\varepsilon)_\varepsilon$
is negligible in
\[
L^\infty(0,T;L^2(\mathbb R^d))
\cap
L^2(0,T;H^1(\mathbb R^d)).
\]

Thus the very weak solution is unique up to negligible perturbations of
the admissible regularising nets.
\end{thm}

\begin{proof}
Set
$w_\varepsilon
:=
u_\varepsilon-\widetilde u_\varepsilon.$
Define
\[
B_\varepsilon(u,v)
:=
\int_{\mathbb R^d}
a_\varepsilon
\nabla u\cdot\overline{\nabla v}\,dx
+
\int_{\mathbb R^d}
b_\varepsilon
(-\Delta)^{s/2}u
\overline{(-\Delta)^{s/2}v}\,dx
+
\int_{\mathbb R^d}
c_\varepsilon u\overline v\,dx,
\]
and define $\widetilde B_\varepsilon$ analogously.

Subtracting the two weak formulations gives
\[
\langle
\partial_t w_\varepsilon,\varphi
\rangle
+
B_\varepsilon(w_\varepsilon,\varphi)
=
(\widetilde B_\varepsilon-B_\varepsilon)
(\widetilde u_\varepsilon,\varphi),
\]
for all $\varphi\in H^1(\mathbb R^d)$, with
$w_\varepsilon(0)
=
u_{0,\varepsilon}
-
\widetilde u_{0,\varepsilon}.$

Taking $\varphi=w_\varepsilon(t)$, we obtain
\begin{align*}
\frac12
\frac{d}{dt}
\|w_\varepsilon(t)\|_{L^2}^2
&+
a_0
\|\nabla w_\varepsilon(t)\|_{L^2}^2
+
b_0
\|(-\Delta)^{s/2}w_\varepsilon(t)\|_{L^2}^2
\\
&+
c_0
\|w_\varepsilon(t)\|_{L^2}^2
\\
&\le
\left|
(\widetilde B_\varepsilon-B_\varepsilon)
(\widetilde u_\varepsilon(t),w_\varepsilon(t))
\right|.
\end{align*}

Furthermore,
\begin{align*}
&
\left|
(\widetilde B_\varepsilon-B_\varepsilon)
(\widetilde u_\varepsilon,w_\varepsilon)
\right|
\\
&\le
\|a_\varepsilon-\widetilde a_\varepsilon\|_{L^\infty}
\|\nabla\widetilde u_\varepsilon\|_{L^2}
\|\nabla w_\varepsilon\|_{L^2}
\\
&\quad+
\|b_\varepsilon-\widetilde b_\varepsilon\|_{L^\infty}
\|(-\Delta)^{s/2}\widetilde u_\varepsilon\|_{L^2}
\|(-\Delta)^{s/2}w_\varepsilon\|_{L^2}
\\
&\quad+
\|c_\varepsilon-\widetilde c_\varepsilon\|_{L^\infty}
\|\widetilde u_\varepsilon\|_{L^2}
\|w_\varepsilon\|_{L^2}.
\end{align*}

By Young's inequality, the coercive terms can be absorbed, giving
\begin{align*}
\frac{d}{dt}
\|w_\varepsilon(t)\|_{L^2}^2
\le C
\Bigl(
&
\|a_\varepsilon-\widetilde a_\varepsilon\|_{L^\infty}^2
\|\nabla\widetilde u_\varepsilon(t)\|_{L^2}^2
\\
&+
\|b_\varepsilon-\widetilde b_\varepsilon\|_{L^\infty}^2
\|(-\Delta)^{s/2}\widetilde u_\varepsilon(t)\|_{L^2}^2
\\
&+
\|c_\varepsilon-\widetilde c_\varepsilon\|_{L^\infty}^2
\|\widetilde u_\varepsilon(t)\|_{L^2}^2
\Bigr),
\end{align*}
where $C>0$ is independent of $\varepsilon$.

Integrating over $(0,T)$ gives
\[
\|w_\varepsilon\|_{L^\infty(0,T;L^2)}^2
\le
\|u_{0,\varepsilon}
-
\widetilde u_{0,\varepsilon}\|_{L^2}^2
+
C
\int_0^T
\Theta_\varepsilon(t)\,dt,
\]
where
\begin{align*}
\Theta_\varepsilon(t)
:=
&
\|a_\varepsilon-\widetilde a_\varepsilon\|_{L^\infty}^2
\|\nabla\widetilde u_\varepsilon(t)\|_{L^2}^2
\\
&+
\|b_\varepsilon-\widetilde b_\varepsilon\|_{L^\infty}^2
\|(-\Delta)^{s/2}\widetilde u_\varepsilon(t)\|_{L^2}^2
\\
&+
\|c_\varepsilon-\widetilde c_\varepsilon\|_{L^\infty}^2
\|\widetilde u_\varepsilon(t)\|_{L^2}^2.
\end{align*}

The coefficient differences are negligible, while the corresponding
solution norms are moderate. Therefore
$\int_0^T\Theta_\varepsilon(t)\,dt$
is negligible.

Together with the negligibility of
$u_{0,\varepsilon}
-
\widetilde u_{0,\varepsilon},$
this shows that
$\|w_\varepsilon\|_{L^\infty(0,T;L^2)}$
is negligible.

Integrating the complete coercive inequality also gives the negligibility of
$\|\nabla w_\varepsilon\|_{L^2(0,T;L^2)}$
and
$\|(-\Delta)^{s/2}w_\varepsilon\|_{L^2(0,T;L^2)}.$

Therefore,
$(w_\varepsilon)_\varepsilon$
is negligible in
$L^\infty(0,T;L^2(\mathbb R^d))
\cap
L^2(0,T;H^1(\mathbb R^d)).$
\end{proof}

\subsection{Consistency with classical theory}
\label{subsec:consistency-space}

We now show that, under suitable convergence assumptions on the
regularisations, the very weak solution is consistent with the classical
weak solution of the regular problem.

\begin{thm}[Consistency]\label{thm:consistency-space}

Assume $0<s<1$ and let
$a,b,c\in L^\infty(\mathbb R^d)$
satisfy
$a(x)\ge a_0>0,
\quad
b(x)\ge b_0>0,
\quad
c(x)\ge c_0>0$
for a.e. $x\in\mathbb R^d$.

Let
$u_0\in H^1(\mathbb R^d),$
and let $u$ be the unique weak solution of \eqref{eq:1.1}.

Let
$(a_\varepsilon,
b_\varepsilon,
c_\varepsilon,
u_{0,\varepsilon})_\varepsilon$
be regularisations satisfying
$a_\varepsilon\to a,
\quad
b_\varepsilon\to b,
\quad
c_\varepsilon\to c$
in $L^\infty(\mathbb R^d)$, and
$u_{0,\varepsilon}\to u_0$
in $H^1(\mathbb R^d)$ as $\varepsilon\to0$.

Assume moreover that
$a_\varepsilon(x)\ge a_0,
\quad
b_\varepsilon(x)\ge b_0,
\quad
c_\varepsilon(x)\ge c_0$
for a.e. $x\in\mathbb R^d$ and every $\varepsilon\in(0,1]$.

Let $u_\varepsilon$ be the unique weak solution of the corresponding
regularised problem. Then
$u_\varepsilon\to u$
in
$C([0,T];L^2(\mathbb R^d))
\cap
L^2(0,T;H^1(\mathbb R^d))$
as $\varepsilon\to0$.
\end{thm}

\begin{proof}
Define
\[
B(v,\varphi)
:=
\int_{\mathbb R^d}
a\nabla v\cdot\overline{\nabla\varphi}\,dx
+
\int_{\mathbb R^d}
b(-\Delta)^{s/2}v
\overline{(-\Delta)^{s/2}\varphi}\,dx
+
\int_{\mathbb R^d}
cv\overline{\varphi}\,dx,
\]
and
\[
B_\varepsilon(v,\varphi)
:=
\int_{\mathbb R^d}
a_\varepsilon\nabla v\cdot\overline{\nabla\varphi}\,dx
+
\int_{\mathbb R^d}
b_\varepsilon(-\Delta)^{s/2}v
\overline{(-\Delta)^{s/2}\varphi}\,dx
+
\int_{\mathbb R^d}
c_\varepsilon v\overline{\varphi}\,dx.
\]

Set
$U_\varepsilon
:=
u-u_\varepsilon.$
For every $\varphi\in H^1(\mathbb R^d)$,
\begin{equation}\label{eq:Ueq-var}
\langle
\partial_tU_\varepsilon,\varphi
\rangle
+
B(U_\varepsilon,\varphi)
=
\langle G_\varepsilon,\varphi\rangle,
\end{equation}
where
\begin{align}
\langle G_\varepsilon,\varphi\rangle
&=
\int_{\mathbb R^d}
(a_\varepsilon-a)
\nabla u_\varepsilon
\cdot
\overline{\nabla\varphi}\,dx
\nonumber\\
&\quad+
\int_{\mathbb R^d}
(b_\varepsilon-b)
(-\Delta)^{s/2}u_\varepsilon
\overline{(-\Delta)^{s/2}\varphi}\,dx
\nonumber\\
&\quad+
\int_{\mathbb R^d}
(c_\varepsilon-c)
u_\varepsilon
\overline{\varphi}\,dx.
\label{eq:Geps-pairing-corr}
\end{align}
Moreover,
$U_\varepsilon(0)
=
u_0-u_{0,\varepsilon}.$

Taking
$\varphi=U_\varepsilon(t)$
in \eqref{eq:Ueq-var}, we obtain
\begin{align}
\frac12
\frac{d}{dt}
\|U_\varepsilon(t)\|_{L^2}^2
&+
\|a^{1/2}\nabla U_\varepsilon(t)\|_{L^2}^2
+
\|b^{1/2}(-\Delta)^{s/2}U_\varepsilon(t)\|_{L^2}^2
\nonumber\\
&+
\|c^{1/2}U_\varepsilon(t)\|_{L^2}^2
=
\langle
G_\varepsilon(t),
U_\varepsilon(t)
\rangle.
\label{eq:energy-Ueps-corr}
\end{align}

By Cauchy--Schwarz and Young's inequality,
\begin{align*}
&
|\langle
G_\varepsilon(t),
U_\varepsilon(t)
\rangle|
\\
&\le
\frac{a_0}{4}
\|\nabla U_\varepsilon(t)\|_{L^2}^2
+
\frac{b_0}{4}
\|(-\Delta)^{s/2}U_\varepsilon(t)\|_{L^2}^2
+
\frac{c_0}{4}
\|U_\varepsilon(t)\|_{L^2}^2
\\
&\quad+
C
\|a_\varepsilon-a\|_{L^\infty}^2
\|\nabla u_\varepsilon(t)\|_{L^2}^2
\\
&\quad+
C
\|b_\varepsilon-b\|_{L^\infty}^2
\|(-\Delta)^{s/2}u_\varepsilon(t)\|_{L^2}^2
\\
&\quad+
C
\|c_\varepsilon-c\|_{L^\infty}^2
\|u_\varepsilon(t)\|_{L^2}^2,
\end{align*}
where $C>0$ is independent of $\varepsilon$.

Using coercivity and integrating over $(0,t)$ yield
\begin{align}
&
\sup_{\tau\in[0,t]}
\|U_\varepsilon(\tau)\|_{L^2}^2
+
\int_0^t
\|\nabla U_\varepsilon(\tau)\|_{L^2}^2\,d\tau
\nonumber\\
&\quad+
\int_0^t
\|(-\Delta)^{s/2}U_\varepsilon(\tau)\|_{L^2}^2\,d\tau
\nonumber\\
&\le
C_T
\Bigl[
\|u_0-u_{0,\varepsilon}\|_{L^2}^2
+
\|a_\varepsilon-a\|_{L^\infty}^2
\|\nabla u_\varepsilon\|_{L^2(0,t;L^2)}^2
\nonumber\\
&\qquad+
\|b_\varepsilon-b\|_{L^\infty}^2
\|(-\Delta)^{s/2}u_\varepsilon\|_{L^2(0,t;L^2)}^2
+
\|c_\varepsilon-c\|_{L^\infty}^2
\|u_\varepsilon\|_{L^2(0,t;L^2)}^2
\Bigr].
\label{eq:Ueps-final-est-corr}
\end{align}

By Lemma~\ref{lem:exist-uniq} and the convergence assumptions,
\[
\sup_{\varepsilon\in(0,1]}
\|u_\varepsilon\|_{L^\infty(0,T;H^1(\mathbb R^d))}
<\infty.
\]
Since
$H^1(\mathbb R^d)
\hookrightarrow
H^s(\mathbb R^d),
\quad
0<s<1,$
we also have
\[
\sup_{\varepsilon\in(0,1]}
\|(-\Delta)^{s/2}u_\varepsilon\|_
{L^\infty(0,T;L^2(\mathbb R^d))}
<\infty.
\]

Therefore,
\[
\|\nabla u_\varepsilon\|_{L^2(0,T;L^2)}
+
\|(-\Delta)^{s/2}u_\varepsilon\|_{L^2(0,T;L^2)}
+
\|u_\varepsilon\|_{L^2(0,T;L^2)}
\le C
\]
with a constant $C>0$ independent of $\varepsilon$.

Since
$u_{0,\varepsilon}\to u_0
\quad
\text{in }H^1(\mathbb R^d)$
and
$a_\varepsilon\to a,
\quad
b_\varepsilon\to b,
\quad
c_\varepsilon\to c
\quad
\text{in }L^\infty(\mathbb R^d),$
the right-hand side of
\eqref{eq:Ueps-final-est-corr}
tends to zero.

Hence
$u_\varepsilon\to u$
in
$L^\infty(0,T;L^2(\mathbb R^d))
\cap
L^2(0,T;H^1(\mathbb R^d)).$

Finally, by Lemma~\ref{lem:exist-uniq},
\[
u,u_\varepsilon
\in
C([0,T];L^2(\mathbb R^d)).
\]
Thus
$U_\varepsilon
=
u-u_\varepsilon
\in
C([0,T];L^2(\mathbb R^d)).$

For a continuous $L^2$-valued function, the supremum and essential
supremum coincide. Consequently,
\[
\|U_\varepsilon\|_{C([0,T];L^2)}
=
\|U_\varepsilon\|_{L^\infty(0,T;L^2)}
\to0.
\]

Therefore,
$u_\varepsilon\to u$
in
$C([0,T];L^2(\mathbb R^d))
\cap
L^2(0,T;H^1(\mathbb R^d)).$
\end{proof}

\begin{rem}\label{rem:consistency-H1sup}
Under the assumptions of
Theorem~\ref{thm:consistency-space}, the above argument gives convergence
in
\[
C([0,T];L^2(\mathbb R^d))
\cap
L^2(0,T;H^1(\mathbb R^d)).
\]
Convergence in
$L^\infty(0,T;H^1(\mathbb R^d))$
would in general require additional regularity assumptions on the
coefficients, the initial datum, or the corresponding solutions.
\end{rem}

\section{Conclusion}\label{4}
In this paper, we studied the Cauchy problem for a parabolic equation driven by a mixed local--nonlocal operator combining a uniformly elliptic diffusion term and a fractional diffusion term, together with a zeroth-order potential. First, under bounded measurable coefficients satisfying uniform positivity assumptions, we proved well-posedness in the weak sense and derived a priori energy estimates in the natural energy space. 

Next, we extended the theory to spatially irregular (distributional) coefficients and compactly supported distributional initial data by adopting a Friedrichs-type regularisation approach. We introduced suitable notions of moderateness and negligibility for the regularising nets, established the existence of very weak solution nets for admissible regularisations, and proved uniqueness modulo negligible perturbations of admissible regularisations. Finally, we showed consistency with the classical weak theory: whenever the data are regular and the regularisations converge appropriately while preserving uniform positivity, the very weak solution converges to the unique weak solution.

The results provide a framework for mixed local--nonlocal parabolic equations with spatial singularities that cannot be treated directly within the standard distributional formulation. Possible extensions include equations with source terms, nonlinear mixed local--nonlocal models, and the development and analysis of numerical approximation schemes for very weak solutions.

\section*{Conflict of interest statement} The authors declare that they have no conflict of interest.

\section*{Funding}
\noindent
The research is financially supported by the FWO Research Grant G083525N: Evolutionary partial differential equations with strong singularities, and by the Methusalem programme of the Ghent University Special Research Fund (BOF) (Grant number 01M01021), also EPSRC grant UKRI3645 and by a grant from the Ministry of Science and Higher Education of the Republic of Kazakhstan (No. AP27508473).

\section*{Author contributions}
\begin{itemize}
\item[-] \textbf{Arshyn Altybay:}
Conceptualization, Investigation, Writing -- original draft,
Writing -- review and editing.

\item[-] \textbf{Michael Ruzhansky:}
Conceptualization, Investigation, Supervision,
Writing -- review and editing.

\end{itemize}


\end{document}